\documentclass{lms}

\usepackage{amsmath,amsfonts,amscd,amssymb}
\usepackage[all]{xy}
\usepackage{color}

\newtheorem{theorem}{Theorem}[section] 
\newtheorem{lemma}[theorem]{Lemma}     
\newtheorem{corollary}[theorem]{Corollary}
\newtheorem{proposition}[theorem]{Proposition}

\newnumbered{assertion}{Assertion}    
\newnumbered{conjecture}{Conjecture}  
\newnumbered{definition}{Definition}
\newnumbered{hypothesis}{Hypothesis}
\newnumbered{remark}{Remark}
\newnumbered{note}{Note}
\newnumbered{observation}{Observation}
\newnumbered{problem}{Problem}
\newnumbered{question}{Question}
\newnumbered{algorithm}{Algorithm}
\newnumbered{example}{Example}
\newunnumbered{notation}{Notation} 
\newunnumbered{examples}{Examples}

\title[Extending valuations to complete local domains.]
{Extending valuations of local domains to complete local domains without changing the value group.}

\author[F.J. Herrera, M. A. Olalla, M. Spivakovsky and B. Teissier]{F.J. Herrera Govantes, M. A. Olalla Acosta, M.
Spivakovsky and B. Teissier}

\classno{12J20 (primary), 16W60, 14B05, 14B25, 13J15 (secondary)}

\extraline{The second author is partially supported by PID2024-156912NB-I00 financied by MICIU/AEI/10.13039/501100011033 and FEDER.}

\begin{document}

\let\noi=\noindent
\let\ss=\underline
\let\sur=\overline
\let\sous=\underline
\newcommand{\bnb}{\sur{B}_{\ss{n}}}
\newcommand{\End}{\textrm{End}}
\newcommand{\id}{\textrm{Id}}
\newcommand{\rg}{\mbox{rg}}
\newcommand{\rk}{\mbox{rk}}
\newcommand{\im}{\mbox{im}}
\newcommand{\ir}{ {\mathcal I} }
\newcommand{\dtr}{\mbox{degtr}}
\def\N{\mathbf{N}} 
\def\Z{\mathbf{Z}} 
\def\R{\mathbf{R}} 
\def\C{\mathbf{C}} 
\def\Q{\mathbf{Q}} 
\def\P{\mathcal{P}}
\def\F{\mathbf{F}}
\def\Tt{\mathcal{T}}
\def\sper{\mbox{Sper}}
\def\spec{\mbox{Spec}}
\def\lead{\mbox{in}}
\def\gr{\mbox{gr}}
\def\init{\mbox{in}}
\def\G{\mbox{l\hspace{-.47em}G}} 
\def\Or{\mathcal{O}}             
\def\notin{\mbox{$\in$ \hspace{-.8em}/}} 
\def\notsubset{\mbox{$\subset$ \hspace{-.9em}/}} 
\def\Ra{\Rightarrow} 
\def\Lra{\Leftrightarrow} 
\def\sgn{\mbox{sgn}} 

\def\g{{\Gamma}}
\def\h{{\Phi}}
\def\md{{\operatorname{mod}}}
\def\he{{\operatorname{ht}}}

\newcommand{\Bernard}[1]{{\color[rgb]{1,0,0} {#1}}}
\newcommand{\Mark}[1]{{\color[rgb]{0,0,1} {#1}}}
\newcommand {\Miguel}[1]{{\color[rgb]{1,0,1} {#1}}}

\maketitle

\begin{abstract}
Let $(R,m,k)$ be an excellent local noetherian domain with field of fractions $K$. Let
$$
\nu:K^*\twoheadrightarrow\Gamma
$$
be a valuation centered at $R$ and let $R_\nu$ be the corresponding valuation ring of $K$, dominating $R$. Denote by $\widehat R$ the $m$-adic completion of $R$. In the applications of valuation theory to commutative algebra and the study of singularities, one is often induced to replace $R$ by its $m$-adic completion $\widehat R$ and $\nu$ by a suitable extension $\widehat\nu_-$ to $\frac{\widehat R}P$ for a suitably chosen prime ideal $P$, such that
$$
P\cap R=(0).
$$
In a previous article we gave a systematic description of all such extensions $\widehat\nu_-$ and defined the notion of tight extensions that are of particular interest for applications (see Herrera, Olalla, Spivakovsky and Teissier, {\em Extending a valuation centered in a local domain to its formal completion}, Proc. London Math. Soc. (3) 105 (2012) 571--621). If $\widehat\nu_-$ is a tight extension then its graded algebra is birational to that of $\nu$ (the converse is not known and might not be true). In particular, the value group of
$\widehat\nu_-$ is $\Gamma$. The existence of tight extensions was conjectured by the last author (see Teissier, {\em Valuations, deformations, and toric geometry},  Fields Institute Communications, {\bf 33}, 2003, 361-459). In the present paper we give a proof of Teissier's conjecture. An intended application of this result is an important step in two recent approaches to local uniformization in positive characteristic.
\end{abstract}


\section{Introduction}
\label{In}

All the rings in this paper will be commutative with 1.

Let $(R,m,k)$ be a local noetherian domain with field of fractions $K$ and $R_\nu$ a valuation ring, birationally dominating $R$. Let $\nu:K^*\twoheadrightarrow\Gamma$ be the corresponding valuation of $K$,  centered at $R$. Here $\Gamma$ is the value group; it is totally ordered and abelian. Let $\widehat R$ denote the $m$-adic completion of $R$. In the applications of valuation theory to commutative algebra and the study of singularities, one is often induced to replace $R$ by its $m$-adic completion $\widehat R$ and $\nu$ by a suitable extension $\widehat\nu_-$ to $\frac{\widehat R}P$ for a suitably chosen prime ideal $P$, such that $P\cap R=(0)$. In the paper \cite{HOST1} we gave, assuming that $R$ is excellent, a systematic description of all such extensions $\widehat\nu_-$ and among them identified the class of tight extensions. The existence of tight extensions (see Conjecture 1 of \cite{HOST1}) is the restatement in the framework of \cite{HOST1} of a conjecture of the last author going back to 2003 (see \cite{Te}, *proposition* 5.19).\par

The main result of the present paper a proof of Teissier's conjecture:

\noindent
\textbf{Given an excellent local domain $R$ and a valuation $\nu$ of $R$ that is strictly positive on its maximal ideal $m$, there exists a prime ideal $P$ of the $m$-adic completion $\widehat R$ such that $P\bigcap R=(0)$ and an extension $\hat\nu_-$ of $\nu$ to $\frac{\widehat R}{P}$ such that the natural inclusion
\begin{equation}
{\rm gr}_\nu R\subset {\rm gr}_{\hat\nu_-}\widehat R/P\label{eq:inclusion}
\end{equation}
of the associated graded rings is scalewise-birational {\rm (see the definitions below)}. In particular, the valuations $\nu$ and $\hat\nu_-$ have the same value group.}\par

The only assumption about $R$ we ever used in \cite{HOST1} is a weaker and more natural condition than excellence: namely, we require the completion homomorphism $R\rightarrow\widehat R$ to be regular. Local rings $R$ having this property are called G-rings or quasi-excellent.

One specific application we have in mind has to do with the approaches to proving the Local Uniformization Theorem in arbitrary characteristic which are described in \cite{Spi2} and \cite{Te}).\par  These approaches to local uniformization both make important use, in different ways, of the structure of the graded ring ${\rm gr}_\nu R$ associated to the filtration of an excellent local domain $R$ defined by the valuation $\nu$. Extending $\nu$  to a valuation $\hat\nu_-$ of a complete local domain $\widehat R/P$ in such a manner that the corresponding extension ${\rm gr}_\nu R\subset {\rm gr}_{\hat\nu_-}\widehat R/P$ of graded rings is scalewise-birational allows one to have access to the many advantages of completeness without losing the algebraicity of the centers of blowing-up (in \cite{Spi2}) or of the embeddings into appropriate larger affine spaces where a birational toric map will provide local uniformization (in \cite{Te}). \par
Let $r$ denote the (real) rank of $\nu$. Let $(0)=\Delta_r\subsetneqq \Delta_{r-1}\subsetneqq\dots\subsetneqq \Delta_0=\Gamma$ be the isolated subgroups of $\Gamma$ and $P_0=(0)\subsetneqq P_1\subseteq\dots\subseteq P_r=m$ the prime valuation ideals of $R$, which need not, in general, be distinct. Assuming that $R$ is excellent, in \cite{HOST1} we canonically associated to $\nu$ a chain of $2r+2$ prime ideals
$$
H_0\subset H_1\subset\dots\subset H_{2r}=H_{2r+1}=m\widehat R,
$$
satisfying $H_{2\ell}\cap R=H_{2\ell+1}\cap R=P_\ell$ and such that $H_{2\ell}$ is a minimal prime of $P_\ell \widehat{R}$ for $0\le\ell\le r$. We called $H_i$ the $i${\bf-th implicit prime ideal} of $\widehat{R}$, associated to $R$ and $\nu$. The role of these ideals is explained at the beginning of section \ref{definition} below.\par

\noindent The ideals $H_i$ behave well under local blowing ups along $\nu$ (that is, birational local homomorphisms $R\to R'$ such that $\nu$ is centered in $R'$). This means that given a local blowing up $R\rightarrow R'$ along $\nu$, the $i$-th implicit prime ideal $H'_i$ of $\widehat{R'}$ has the property that $H'_i\cap \widehat{R}=H_i$. In this situation we say that the collection $\{H'_i\}$ forms a {\bf tree} (see subsection \ref{trees}).

For a prime ideal $P$ in a ring $R$, $\kappa(P)$ will denote the residue field $\frac{R_P}{PR_P}$.

Let $(0)=\mathbf{m}_0\subsetneqq \mathbf{m}_1\subsetneqq\dots\subsetneqq\mathbf{m}_{r-1}\subsetneqq\mathbf{m}_r=\mathbf{m}_\nu$ be the prime ideals of the valuation ring $R_\nu$. By definitions, our valuation $\nu$ is a composition of $r$ rank one valuations\textcolor{red}{,} $\nu=\nu_1\circ\nu_2\dots\circ\nu_r$, where $\nu_\ell$ is a valuation of the field $\kappa(\mathbf{m}_{\ell-1})$, centered at $\frac{(R_\nu)_{\mathbf{m}_\ell}}{\mathbf{m}_{\ell -1}(R_\nu)_{\mathbf{m}_\ell}}$. In particular, $\nu_\ell$ defines a valuation on $\frac{R}{P_{\ell -1}}$ with value group $\frac{\Delta_{\ell -1}}{\Delta_{\ell }}$ by restriction (see \cite{ZS}, Chapter VI, \S10, p. 43 for the definition of composition of valuations; more information is given below in subsection \ref{trees}, where we interpret each $\mathbf{m}_\ell$ as the limit of a tree of ideals). For each integer
$\ell\in\{1,\dots,r\}$, let $\mu_\ell:=\nu_\ell\circ\nu_{\ell+1}\circ\dots\circ\nu_r$; it is the residual valuation induced by $\nu$ on $\frac{R}{P_{\ell -1}}$.

\subsection{Trees.}\label{trees}

We consider birational $\nu$-\textit{extensions} $R\rightarrow R'$ of local rings, that is, injective homomorphisms such that $R'$ is an $R$-algebra essentially of finite type with field of fractions $K$, dominated by $R_\nu$ (in particular, we have $m'\cap R=m$). Such extensions form a direct system $\{R'\}$ with
\begin{equation}
\lim\limits_{\overset\longrightarrow{R'}}R'=R_\nu.\label{eq:ZariskiRiemann}
\end{equation}
We will consider many direct systems of rings and of ideals indexed by $\{R'\}$; direct limits will always be taken with respect to the direct system $\{R'\}$.

\begin{definition} A \textbf{tree} of $R'$-algebras is a direct system $\{S'\}$ of rings, indexed by the directed set $\{R'\}$, where $S'$ is an $R'$-algebra. Note that the tree maps are not required to be  injective or birational. A morphism $\{S'\}\to\{T'\}$ of trees is the datum of a map of $R'$-algebras $S'\to T'$ for each $R'$ commuting with the tree morphisms $S'\to S''$ and $T'\to T''$ for each map $R'\to R''$.
\end{definition}
\begin{definition} Let $\{S'\}$ be a tree of $R'$-algebras. For each $S'$, let $I'$ be an ideal of $S'$. We say that $\{I'\}$ is a {\bf tree of ideals} if for every arrow $b_{S'S''}\colon S'\rightarrow S''$ in our direct system, we have $b^{-1}_{S'S''}I''=I'$. We have the obvious notion of inclusion of trees of ideals. In particular, we may speak about chains of trees of ideals.
\end{definition}

\begin{examples} For each non-negative element $\beta\in\Gamma$, the valuation ideals $\P'_\beta\subset R'$ of value $\beta$ form a tree of ideals of $\{R'\}$. Similarly, the $i$-th prime valuation ideals $P'_i\subset R'$ form a tree. If $rk\ \nu=r$, the prime valuation ideals $P'_i$ give rise to a chain
\begin{equation}
(0)=P'_0\subsetneqq P'_1\subseteq\dots\subseteq P'_r=m'\label{eq:treechain'}
\end{equation}
of trees of prime ideals of $\{R'\}$.
\medskip

In \cite{HOST1} we proved that the implicit prime ideals $H'_i$ form a tree of ideals of $\widehat{R}$.
\end{examples}
We showed that specifying an extension $\widehat\nu_-$ of $\nu$ as above is equivalent to specifying a chain
\begin{equation}
\tilde H'_0\subset\tilde H'_1\subset\dots\subset\tilde H'_{2r}=m'\widehat R'\label{eq:chaintree}
\end{equation}
of trees of prime valuation ideals of $\widehat{R'}$ such that $H'_\ell\subset\tilde H'_\ell$ for all $\ell\in\{0,\dots,2r\}$, and valuations
$\widehat\nu_1,\widehat\nu_2,\dots,\widehat\nu_{2r}$, where $\widehat\nu_i$ is a valuation of the field $k_{\widehat\nu_{i-1}}$ (the residue field of the valuation ring $R_{\widehat\nu_{i-1}}$), arbitrary when $i$ is odd and satisfying certain conditions, coming from the valuation $\nu_{\frac i2}$, when $i$ is even.
\medskip

\noindent Recall that for each $\beta\in \Gamma$ one associates to $\beta$ the valuation ideals
$$
\begin{array}{lr}
{\cal P}_\beta (R)=&\{x\in R/\nu(x)\geq\beta\}\cr
{\cal P}^+_\beta (R)=&\{x\in R/\nu(x)>\beta\}
\end{array}
$$
of $R$ (with the convention that $\nu(0)=\infty$, an element larger than any element of $\Gamma$) and the graded ring
$$
\hbox{\rm gr}_\nu R=\bigoplus_{\beta\in\Gamma_+}\frac{{\cal P}_\beta (R)}{{\cal P}^+_\beta (R)}.
$$
Note that if the ring $R$ is noetherian, the semigroup $\Phi \subset \Gamma_{\geq 0}$ of values of $\nu$ on $R\setminus\{0\}$ is well ordered. If $R\to R'$ is a birational extension of local domains dominated by $R_\nu$, we will write ${\cal P}'_\beta$ for ${\cal P}_\beta(R')$.

\begin{notation} Let ${\cal T}=\{R'\}$, the tree defined above. For a ring $R'\in\cal T$, we shall denote by ${\cal T}(R')$ the subtree of $\cal T$ consisting of all the
$\nu$-extensions $R''$ of $R'$.
\end{notation}
Let
\begin{equation}
\Gamma\hookrightarrow\widehat\Gamma\label{eq:extGamma}
\end{equation}
be an extension of ordered groups of the same rank. Let
$$
(0)=\widehat\Delta_r\subsetneqq\widehat\Delta_{r-1}\subsetneqq\dots\subsetneqq\widehat\Delta_0=\widehat\Gamma
$$
be the isolated subgroups of $\widehat\Gamma$, so that the inclusion (\ref{eq:extGamma}) induces inclusions
\begin{eqnarray}
\Delta_\ell&\hookrightarrow&\widehat\Delta_\ell\quad\text{ and}\\
\frac{\Delta_\ell}{\Delta_{\ell+1}}&\hookrightarrow&\frac{\widehat\Delta_\ell}{\widehat\Delta_{\ell+1}}.
\end{eqnarray}

Let $G\hookrightarrow\widehat G$ be an extension of graded algebras without zero divisors, such that $G$ is graded by  $\Gamma_+$ and $\widehat G$ by
$\widehat\Gamma_+$. The graded algebra $G$ is endowed with a natural valuation with value group $\Gamma$ and similarly for $\widehat G$ and $\widehat\Gamma$. Both of these natural valuations will be denoted by $ord$.

\begin{definition}\label{scalewiseDef} We say that the extension $G\hookrightarrow\widehat G$ is \textbf{scalewise-birational} if $G_0=\widehat G_0$ and for all homogeneous elements $x\in\widehat G$ and $\ell\in\{1,\dots,r-1\}$ such that $ord\ x\in\widehat\Delta_\ell$ there exists $y\in G$ such that $ord\ y\in\Delta_\ell$ and $xy\in G$.
\end{definition}
Consider a tree of prime ideals $H'$ of $\widehat{R'}$ with $H'\cap R'=(0)$ and a valuation $\widehat\nu_-$, centered at $\lim\limits_\to\frac{\widehat{R'}}{H'}$.
\begin{definition} We say that $\hat\nu_-$ is a \textbf{scalewise-birational extension of} $\nu$ if for each $R'$ the inclusion
\begin{equation}
{\rm gr}_\nu R'\subset {\rm gr}_{\hat\nu_-}\widehat{R'}/H'\label{eq:inclusionR'}
\end{equation}
is scalewise-birational.
\end{definition}
Of course, if  $\hat\nu_-$ is a scalewise-birational extension of $\nu$ then the inclusion \eqref{eq:inclusionR'} is birational in the usual sense (that is, induces equality of fraction fields) which, in turn, implies that $\widehat\Gamma=\Gamma$.\par
\medskip

The main result of the present paper is:\par
\begin{theorem}\label{teissier} There exists a tree of prime ideals $\tilde H'_0$ of $\widehat{R'}$ with $\tilde H'_0\cap R'=(0)$ and a scalewise-birational extension
$\widehat\nu_-$, centered at $\lim\limits_\to\frac{\widehat{R'}}{\tilde H'_0}$.
\end{theorem}
The results of \cite{Cut} and the example given in Remark 5.20, 4) of \cite{Te} shows that the morphism of associated graded rings is not an isomorphism in general.
\medskip

\noindent{\bf Acknowledgement.} We acknowledge earlier work in this area by W. Heinzer and J. Sally \cite{HeSa} (which deals with the case $\dim\ R=2$) as well as by the school of S. D. Cutkosky \cite{CE,CG} that helped inspire some of the authors to think about the subject.

\section{A review of some of the main definitions and results from [7]}

\subsection{Definition and first properties of implicit prime ideals}\label{definition}

Let $0\le\ell\le r$. We recall the definition of one of the main objects in the theory, the \textit{implicit prime ideals} of $\nu$ in $\widehat R$. Broadly speaking, the odd-numbered implicit prime ideals represent the fact that elements of $\widehat R$ can be limits of sequences of elements of $R$ of ever increasing valuation, tending to infinity at least for one of the valuations with which $\nu$ is composed, while even-numbered implicit ideals represent the fact that while some quotients $\frac{R}{P_\ell}$ may not be analytically irreducible, the valuation $\nu$ ``chooses" one of the irreducible components of the completion. These representations must then be ``stabilized'' with respect to the tree
${\mathcal T}$ in the way shown by equation (\ref{eq:defin}) below.\par
The $(2\ell+1)$-st {\bf implicit prime ideal} $H_{2\ell+1}\subset \widehat{R}$ is defined by
\begin{equation}
H_{2\ell+1}=\bigcap\limits_{\beta\in\Delta_\ell}\left(\left(\lim\limits_{\overset\longrightarrow{R'}}{\cal P}'_\beta\widehat{R'}\right)\bigcap\widehat{R}\right),\label{eq:defin}
\end{equation}
where $R'$ ranges over $\mathcal{T}$. As usual, we think of (\ref{eq:defin}) as a tree equation: if we replace $R$ by any other $R''\in{\cal T}$ in (\ref{eq:defin}), it defines the corresponding ideal $H''_{2\ell +1}\subset\widehat{R''}$.
\begin{proposition}\label{contracts} (\cite{HOST1}, Proposition 3.1) We have $H'_{2\ell+1}\cap R'=P'_\ell$.
\end{proposition}
\par\noindent The implicit ideal $H_{2\ell}$ is defined as the unique minimal prime ideal of $P_\ell\widehat R$ contained in $H_{2\ell+1}$.\par
Here the uniqueness follows from the fact that $R$ being excellent (or a $G$-ring), the ring $\widehat R\otimes_R\kappa(P_\ell)$ is regular, hence a domain. Thus the only minimal prime of $\widehat R\otimes_R\kappa(P_\ell)$ is $(0)$ and hence the unique minimal prime of $P_\ell\widehat R$ in $\widehat R$ is the kernel of the natural map
$\widehat R\rightarrow\widehat R\otimes_R\kappa(P_\ell)$.
\par\medskip

In (\cite{HOST1}, Theorem 8.1) we prove the primality of the implicit ideals.

\subsection{A classification of extensions of $\nu$ to $\widehat R$.}
\label{Rdag}

One is naturally led to consider the more general problem of extending $\nu$ not only to rings of the form $\frac{\widehat{R}}P$ but also to the ring
$\lim\limits_{\overset\longrightarrow{R'}}\frac{\widehat{R'}}{P'}$, where $P'$ is a tree of prime ideals of $\widehat{R'}$, such that $P'\cap R'=(0)$.

In \S5 of \cite{HOST1}, we gave a systematic description of all the possible extensions $\widehat{\nu}_-$ of $\nu$ to
$\lim\limits_{\overset\longrightarrow{R'}}\frac{\widehat{R'}}{P'\widehat{R'}}$, as compositions of $2r$ valuations,
\begin{equation}
\widehat{\nu}_-=\widehat{\nu}_1\circ\dots\circ\widehat{\nu}_{2r},\label{eq:composition}
\end{equation}
satisfying certain conditions. We associate to each extension $\widehat{\nu}_-$ of $\nu$  a chain
\begin{equation}
\tilde H'_0\subset\tilde H'_1\subset\dots\subset\tilde H'_{2r}=m'\widehat{R'}\label{eq:chaintree''}
\end{equation}
of prime $\widehat{\nu}_-$-ideals, corresponding to the decomposition (\ref{eq:composition}), satisfying
\begin{equation}
\tilde H'_{2\ell}\cap R'=\tilde H'_{2\ell+1}\cap R'=P'_\ell.\label{eq:tildeHcapR}
\end{equation}.

By definitions, for $1\le i\le2r$, $\widehat{\nu}_i$ is a valuation of the field $k_{\widehat{\nu}_{i-1}}$.
\begin{proposition}\label{Hintilde} (\cite{HOST1}, Proposition 5.3) We have
\begin{equation}
H'_i\subset\tilde H'_i\qquad\text{ for all }i\in\{0,\dots,2r\}.\label{eq:Hintilde}
\end{equation}
\end{proposition}

One of the main theorems of \cite{HOST1} (Theorem 5.4) says that specifying the valuation $\widehat{\nu}_-$ is equivalent to specifying the following data:

(1) A chain of trees (\ref{eq:chaintree''}) of prime ideals of $\widehat{R'}$ satisfying \eqref{eq:Hintilde} and some additional conditions.

(2) For each $i$, $1\le i\le 2r$, a valuation $\widehat{\nu}_i$ of $k_{\widehat{\nu}_{i-1}}$ (where $\widehat{\nu}_0$ is taken to be the trivial valuation by convention), whose restriction to $\lim\limits_{\overset\longrightarrow{R'}}\kappa(\tilde H'_{i-1})$ is centered at the local ring $\lim\limits_{\overset\longrightarrow{R'}}\frac{\widehat{R'}_{\tilde H'_i}}{\tilde H'_{i-1}\widehat{R'}_{\tilde H'_i}}$.

The data $\left\{\widehat{\nu}_i\right\}_{1\le i\le 2r}$ is subject to the following additional condition: if $i=2\ell$ is even then $rk\ \widehat{\nu}_i=1$ and $\widehat{\nu}_i$ is an extension of $\nu_i$ to $k_{\widehat{\nu}_{i-1}}$ (which is naturally an extension of $k_{\nu_{i-1}}$).

In particular, such extensions $\widehat{\nu}_-$ always exist, and usually there are plenty of them.

\subsection{Tight extensions and scalewise birationality}

In \S6 of \cite{HOST1} we define and study, among other things, a class of extensions $\widehat{\nu}_-$ which are particularly useful for the applications called the \textbf{tight} extensions (see \cite[Definition 6.1]{HOST1}). One of the properties of tight extensions is the fact that
\begin{equation}
\tilde H'_{2\ell}=\tilde H'_{2\ell+1}\quad\text{ for }1\le\ell\le r.\label{eq:odd=even}
\end{equation}
\begin{proposition}\label{tight=scalewise} (\cite{HOST1}, Proposition 6.7) The extension $\widehat{\nu}_-$ is tight if and only if for each $R'$ in our direct system the natural graded algebra extension $\gr_\nu R'\rightarrow\gr_{\widehat{\nu}_-}\widehat{R'}$ is scalewise-birational.
\end{proposition}
\begin{remark}\label{rephrasing} Proposition \ref{tight=scalewise} allows us to rephrase Conjecture 1 of \cite{HOST1} as follows: the valuation $\nu$ admits at least one tight extension $\widehat{\nu}_-$.
\end{remark}
The only material from \cite{HOST1} used in the proof of the main theorem of the present paper is the definition \eqref{eq:defin} of implicit ideals, their primality (Theorem 8.1) and the equality
$$
H'_{2\ell+1}\cap R'=P'_\ell
$$
(Proposition 3.1), but not the somewhat technical Theorem 5.4.

\section{A proof of the existence of a scalewise-birational extension $\hat\nu_-$}\label{locuni1}

The purpose of this section is to prove Theorem \ref{teissier} by constructing a valuation $\widehat\nu_-$ whose associated graded algebra is a scalewise-birational extension of
$\gr_\nu R$.

The construction consists of describing the trees of ideals $\tilde H'_i$, $0\le i\le 2r$ and, for each $i$, a valuations $\widehat\nu_i$ of the residue field $k_{\hat\nu_{i-1}}$, such that $\widehat\nu_-=\widehat\nu_1\circ\dots\widehat\nu_{2r}$. More precisely, for $\ell\in\{1,\dots,r\}$, we will construct, recursively in the descending order of $\ell$, a tree
$\tilde H'_{2\ell-1}$ of prime ideals of $\widehat{R'}$, $R'\in\mathcal{T}$, such that $\tilde H'_{2\ell-1}\cap R'=P'_{\ell-1}$, and a  scalewise-birational extension
$\widehat\mu_{2\ell}$ of $\mu_\ell$ to $\lim\limits_{\overset\longrightarrow{R'\in\mathcal{T}}}\frac{\widehat{R'}}{\tilde H'_{2\ell-1}}$. The valuation $\widehat\mu_2$ will be the desired scalewise-birational extension $\widehat\nu_-$ of $\mu_1=\nu$.

We will have $\tilde H'_{2\ell}=\tilde H'_{2\ell+1}$.

Let $\widehat{\mathbf R}:=\lim\limits_{\overset\longrightarrow{R'\in\mathcal{T}}}\widehat {R'}$ and $\mathbf H_i=\lim\limits_{\overset\longrightarrow{R'\in\mathcal{T}}}\tilde H'_i$. Then $\hat\nu_{2\ell+1}$ will be the trivial valuation of the field $\frac{\widehat{\bf R}_{\mathbf H_{2\ell+1}}}{\mathbf H_{2\ell}\widehat{\bf R}_{\mathbf H_{2\ell+1}}}$.
\bigskip

We define $\tilde H'_{2r}=\tilde H'_{2r+1}=m'\widehat{R'}$ and let $\widehat\mu_{2r+2}=\widehat\nu_{2r+2}$ be the trivial valuation of the residue field $k'$.

Next, assume that $\ell\in\{1,\dots,r\}$, that the tree $\left\{\tilde H'_{2\ell+1}\right\}\supset\left\{H'_{2\ell+1}\right\}$ of prime ideals of $\widehat{R'}$ and a
scalewise-birational extension $\widehat\mu_{2\ell+2}$ of $\mu_{\ell+1}$ to $\lim\limits_{\overset\longrightarrow{R'}}\frac{\widehat{R'}}{\tilde H'_{2\ell+1}}$ are already constructed for $R'\in\mathcal{T}$ and that
$$
\tilde H'_{2\ell+1}\cap R'=P'_\ell.
$$
Define $\tilde H'_{2\ell}:=\tilde H'_{2\ell+1}$ and let $\widehat\nu_{2\ell+1}$ be the trivial valuation  of the field $\frac{\widehat{\bf R}_{\mathbf H_{2\ell+1}}}{\mathbf
H_{2\ell}\widehat{\bf R}_{\mathbf H_{2\ell+1}}}$.

It remains to construct the ideals $\tilde H'_{2\ell-1}\subset\widehat{R'}$ and a scalewise-birational extension $\widehat\mu_{2\ell}$ of $\mu_\ell$ to
$\frac{\widehat{R'}}{\tilde H'_{2\ell-1}}$ for $R'$ in $\mathcal{T}$.

Consider the set $\mathcal I$ of all the trees of prime ideals $\{H'\}$ of $\left\{\widehat R'\right\}$ contained in $\tilde H'_{2\ell}$ and satisfying the condition
\begin{equation}
H'\cap R'=P'_{\ell-1}. \label{eq:restrictionmain0}
\end{equation}
The collection $\mathcal I$ is not empty since the tree $\{H'_{2\ell-1}\}$ belongs to it by Proposition \ref{contracts} (Proposition 3.1 of \cite{HOST1}). Let $\left\{\tilde
H'_{2\ell-1}\right\}$ be a maximal (in the sense of inclusion) element of $\mathcal I$. Here by ``maximal'' we mean that there is no tree $\{I'\}$ in $\mathcal I$ containing  $\left\{\tilde H'_{2\ell-1}\right\}$ such that $\left\{\tilde H'_{2\ell-1}\right\}\subsetneqq I'$ for some $I'$ in $\mathcal I$ and hence for all $I'$ that are sufficiently far out  in the tree. Such a maximal tree $\left\{\tilde H'_{2\ell-1}\right\}$ exists by Zorn's Lemma.

By the induction assumption, the valuation $\widehat\mu_{2\ell+2}$ is already defined on $\frac{\widehat{\bf R}}{\bf H_{2\ell+1}}$ and, for all $R'\in\mathcal{T}$, the inclusion
$\gr_{\mu_{\ell+1}}\frac{R'}{P'_\ell}\subset\gr_{\widehat\mu_{2\ell+2}}\frac{\widehat{R'}}{\tilde H'_{2\ell+1}}$ of graded algebras is scalewise-birational.

It remains to describe the valuation $\widehat\mu_{2\ell}$ centered in $\frac{\widehat{R'}}{\tilde H'_{2\ell-1}}$, $R'\in\mathcal T$. We will first define $\widehat\nu_{2\ell}$ and then define $\widehat\mu_{2\ell}$ as $\widehat\nu_{2\ell}\circ\widehat\mu_{2\ell+2}$.

\begin{lemma} Let $\mu$ be a rank one valuation, centered in a local noetherian domain $(A,M)$. Let
$$
\Psi=\mu(A\setminus\{0\}).
$$
Then $\Psi$ contains no infinite strictly increasing bounded sequences.
\end{lemma}
\begin{proof} For an element $\beta\in\Psi$, let $\P_\beta:=\P_\beta(A)$ denote the $\mu$-ideal of $A$ of value $\beta$. An infinite ascending sequence
$\alpha_1<\alpha_2<\dots$ in $\Psi$, bounded above by an element $\beta\in\Psi$, would give rise to an infinite descending chain of ideals in $\frac A{\P_\beta}$. Thus it is sufficient to prove that $\frac A{\P_\beta}$ has finite length.

Let $\delta:=\mu(M)(=\min(\Psi\setminus\{0\}))$. Since $\rk\ \mu=1$, there exists $n\in\N$ such that $\beta\le n\delta$. Then $M^n\subset\P_\beta$, so there is a surjective homomorphism $\frac A{M^n}\twoheadrightarrow\frac A{\P_\beta}$. Thus $\frac A{\P_\beta}$ has finite length, as desired.
\end{proof}

\begin{remark}See \cite[Theorem 3.2]{CT} for a more general result about the absence of accumulation points in semigroups.\end{remark}

\begin{corollary}\label{wellordered} \Mark The semigroup $\nu\left(\frac{R'_{P'_\ell}}{P'_{\ell-1}}\setminus\{0\}\right)$ is isomorphic to $\mathbf N$ as an ordered set.

\end{corollary}
\begin{proof} Since $\nu_\ell$ is a rank 1 valuation, centered in the local noetherian domain $\frac{R'_{P'_\ell}}{P'_{\ell-1}}$, \Mark the semigroup 
$\nu\left(\frac{R'_{P'_\ell}}{P'_{\ell-1}}\setminus\{0\}\right)$ contains no infinite strictly increasing bounded sequences. The corollary now follows from the fact that
$\nu\left(\frac{R'_{P'_\ell}}{P'_{\ell-1}}\setminus\{0\}\right)$ is well ordered.
\end{proof}

\begin{lemma}\label{crucial} Fix a ring $R'$ in $\mathcal T$ and a non-zero element $x\in\frac{\widehat{R'}_{\tilde H'_{2\ell}}}{\tilde H'_{2\ell-1}\widehat{R'}_{\tilde H'_{2\ell}}}$. Then
\[
(x)\bigcap\frac{R'_{P'_\ell}}{P'_{\ell-1}}\ne(0).
\]
\end{lemma}
\begin{proof} Let $Q_1,\dots,Q_s$ be the set of minimal primes of  $(x)$ in $\frac{\widehat{R'}_{\tilde H'_{2\ell}}}{\tilde H'_{2\ell-1}\widehat{R'}_{\tilde H'_{2\ell}}}$. It is sufficient to prove that 
\begin{equation}
Q_i\cap\frac{R'_{P'_\ell}}{P'_{\ell-1}}\ne(0)\quad\text{for all }\quad i\in\{1,\dots,s\}.\label{eq:intnonzero}
\end{equation}
Indeed, once \eqref{eq:intnonzero} is established, let $y_i$ be a non-zero element of $Q_i\cap\frac{R'_{P'_\ell}}{P'_{\ell-1}}$. Then there exists $N\in\bf N$ such that
$0\ne\left(\prod\limits_{i=1}^sy_i\right)^N\in(x)\bigcap\frac{R'_{P'_\ell}}{P'_{\ell-1}}$, as desired.

To prove \eqref{eq:intnonzero}, we  argue by contradiction. Assume that $\left(x\right)$ has a minimal prime $Q'$ in $\frac{\widehat{R'}_{\tilde H'_{2\ell}}}{\tilde
H'_{2\ell-1}\widehat{R'}_{\tilde H'_{2\ell}}}$ satisfying
\[
Q'\cap\frac{R'_{P'_\ell}}{P'_{\ell-1}}=(0).
\]
For every $\nu$-extension $R'\to R''$, there exists a minimal prime $Q''$ of $\left(x\right)$ in $\frac{\widehat{R''}_{\tilde H''_{2\ell}}}{\tilde H''_{2\ell-1}\widehat{R''}_{\tilde H''_{2\ell}}}$ lying over $Q'$ and satisfying
\begin{equation}
Q''\cap\frac{R''_{P''_\ell}}{P''_{\ell-1}}=(0).\label{eq:intersectionzero}
\end{equation}
By Zorn's lemma, there exists a tree $\{P''\}$ of prime ideals of $\left\{\frac{\widehat{R''}_{\tilde H''_{2\ell}}}{\tilde H''_{2\ell-1}\widehat{R''}_{\tilde H''_{2\ell}}}\right\}$ satisfying \eqref{eq:intersectionzero}. Taking the preimage of $\{P''\}$ in $\left\{\widehat{R''}\right\}$ produces a tree $\{H''\}$ satisfying \eqref{eq:restrictionmain0} and properly containing the tree $\left\{\tilde H''_{2\ell-1}\right\}$, contradicting the maximality of $\left\{\tilde H''_{2\ell-1}\right\}$. This completes the proof of the lemma.
\end{proof}

For a positive element $\bar\beta\in\frac{\Delta_{\ell-1}}{\Delta_\ell}$ denote by $\widehat{\mathcal P}'_{\bar\beta ,\ell}$ the ideal
\begin{equation}
\widehat{\mathcal P}'_{\bar\beta ,\ell}=\frac{\left(\P'_{\bar\beta}+\tilde H'_{2\ell-1}\right)\widehat{R'}_{\tilde H'_{2\ell}}}{\tilde H'_{2\ell-1}\widehat{R'}_{\tilde H'_{2\ell}}}.\label{eq:validealmain}
\end{equation}

Since by the definition of $\tilde H'_{2\ell-1}$ we have $\bigcap_{\bar\beta}\widehat{\mathcal P}'_{\bar\beta ,\ell}=(0)$, we can define the pseudo-valuation
$\widehat\nu_{2\ell}$ centered at  $\frac{\widehat{R'}_{\tilde H'_{2\ell}}}{\tilde H'_{2\ell-1}\widehat{R'}_{\tilde H'_{2\ell}}}$ by
\begin{equation}\label{eq:mu2l}
\widehat\nu_{2\ell}(x):=\max\left\{\left.\bar\beta\in\left(\frac{\Delta_{\ell-1}}{\Delta_\ell}\right)_+\ \right|\ x\in\widehat{\mathcal P}'_{\bar\beta ,\ell}\right\}.
\end{equation}
\
\begin{lemma} We have
\begin{equation}
\widehat\P'_{\bar\beta}\bigcap\frac{R'_{P'_\ell}}{P'_{\ell-1}}=\P'_{\bar\beta}.\label{eq:Jnotempty}
\end{equation}
\end{lemma}
\begin{proof} The inclusion $\supset$ is clear from definitions. Let us prove the opposite inclusion.

By Zorn's lemma, there exist extensions $\tilde\nu$ of $\nu_\ell$ to valuations centered in the ring $\frac{\widehat{R'}_{\tilde H'_{2\ell}}}{\tilde H'_{2\ell-1}\widehat{R'}_{\tilde
H'_{2\ell}}}$. Fix one such extension $\tilde\nu$ and let $\tilde\P'_{\bar\beta}$ denote the greatest $\tilde\nu$-ideal of $\frac{\widehat{R'}_{\tilde H'_{2\ell}}}{\tilde
H'_{2\ell-1}\widehat{R'}_{\tilde H'_{2\ell}}}$ among those of value at least $\bar\beta$. We have
\begin{equation}
\widehat{\mathcal P}'_{\bar\beta ,\ell}\subset\tilde\P'_{\bar\beta},\label{eq:hatintilde}
\end{equation}
so $\widehat{\mathcal P}'_{\bar\beta ,\ell}\bigcap\frac{R'_{P'_\ell}}{P'_{\ell-1}}\subset\tilde\P'_{\bar\beta}\bigcap\frac{R'_{P'_\ell}}{P'_{\ell-1}}=\P'_{\bar\beta}$, which proves the desired inclusion.
\end{proof}

\begin{corollary} The restriction of the pseudo-valuation $\widehat\nu_{2\ell}$ to $\frac{R'_{P'_\ell}}{P'_{\ell-1}}$ coincides with $\nu_\ell$.
\end{corollary}

\begin{remark}\label{tildelessthanhat} The inclusion \eqref{eq:hatintilde} is equivalent to saying that for every $x\in\frac{\widehat{R'}_{\tilde H'_{2\ell}}}{\tilde
H'_{2\ell-1}\widehat{R'}_{\tilde H'_{2\ell}}}$ we have $\tilde\nu(x)\le\widehat\nu_{2\ell}(x)$. The next lemma shows that this inequality is, in fact, an equality.
\end{remark}

Keep the notation of the above remark. Let $\bar\beta=\widehat\nu_{2\ell}(x)$.

\begin{lemma} We have
\begin{equation}\label{eq:mu2ltilde}
\tilde\nu(x)=\bar\beta.
\end{equation}
\end{lemma}
\begin{proof} Take a non-zero element $y\in(x)\bigcap\frac{R'_{P'_\ell}}{P'_{\ell-1}}$ (such an element exists by Lemma \ref{crucial}). Let $u\in\frac{\widehat{R'}_{\tilde
H'_{2\ell}}}{\tilde H'_{2\ell-1}\widehat{R'}_{\tilde H'_{2\ell}}}$ be such that $y=ux$. Since $\widehat\nu_{2\ell}$ is a pseudo-valuation extending $\nu_\ell$, we have
\begin{equation}
\nu_\ell(y)=\widehat\nu_{2\ell}(y)=\widehat\nu_{2\ell}(xu)\ge\bar\beta+\widehat\nu_{2\ell}(u)\ge\tilde\nu(x)+\tilde\nu(u)=\tilde\nu(y)=\nu_\ell(y),
\label{eq:inequalityisequality}
\end{equation}
where the first inequaliity is given by the ultrametric triangle rule and the second follows from Remark \ref{tildelessthanhat}, applied separately to $x$ and $u$. Thus all the inequalities in \eqref{eq:inequalityisequality} are equalities and the result follows.
\end{proof}

\begin{corollary}
\begin{enumerate}
\item The pseudo-valuation $\widehat\nu_{2\ell}$ is, in fact, a valuation.
\item The valuation $\widehat\nu_{2\ell}$ is the unique valuation extending $\nu_\ell$ to $\frac{\widehat{R'}_{\tilde H'_{2\ell}}}{\tilde H'_{2\ell-1}\widehat{R'}_{\tilde
H'_{2\ell}}}$.
\end{enumerate}
\end{corollary}

\begin{remark} Since the construction above is valid for all $R''$ in $\mathcal T$, $\widehat\nu_{2\ell}$ extends naturally to a valuation $\widehat\nu_{2\ell}:\frac{\widehat{\bf
R}_{\mathbf H_{2\ell}}}{\bf H_{2\ell+1}}\setminus\{0\}\rightarrow\frac{\Delta_{\ell-1}}{\Delta_\ell}$ (which we still denote by $\widehat\nu_{2\ell}$). The valuation
$\widehat\nu_{2\ell}$ is centered in the local ring $\frac{\widehat{\bf R}_{\mathbf H_{2\ell}}}{\bf H_{2\ell+1}}$. Moreover, $\widehat\nu_{2\ell}$ is the unique valuation extending
$\nu_\ell$ with this property.
\end{remark}

Let
$$
\mathbf P_{\bar\beta}:=\lim\limits_{\overset\longrightarrow{R'\in\mathcal{T}}}\frac{\P'_{\bar\beta}R'_{P'_\ell}}{P'_{\ell-1}R'_{P'_\ell}},
$$
$$
\widehat{\mathbf P}_{\bar\beta,\ell}:=\lim\limits_{\overset\longrightarrow{R'\in\mathcal{T}}}\widehat{\P'}_{\bar\beta,\ell},
$$
and similarly for $\mathbf P^+_{\bar\beta}$ and $\widehat{\mathbf P}^+_{\bar\beta,\ell}$. Recall that
$$
\mathbf m_\ell=\lim\limits_{\overset\longrightarrow{R'\in\mathcal{T}}}P'_\ell.
$$
By definitions (see \eqref{eq:validealmain}), for all $\bar\beta\in\left(\frac{\Delta_{\ell-1}}{\Delta_\ell}\right)_+$, we have a natural surjective homomorphism
\begin{equation}
\lambda_{\bar\beta}:\frac{\P'_{\bar\beta}}{{\P'}^+_{\bar\beta}}\otimes_{\kappa\left(P'_{\ell}\right)}\kappa\left(\tilde
H'_{2\ell}\right)\longrightarrow \frac{\widehat{\mathcal P'}_{\bar\beta,\ell}}{\widehat{\mathcal{P'}}^+_{\bar\beta,\ell}}\label{eq:treeisomorphism}
\end{equation}
of $\kappa\left(\tilde H'_{2\ell}\right)$-vector spaces. Passing to direct limits as $R'$ runs over the tree $\mathcal T$, we obtain a natural surjective homomorphism
\begin{equation}\label{eq:limitisomorphism}
\lambda_{\bar\beta}:\frac{\mathbf P_{\bar\beta}}{\mathbf P^+_{\bar\beta}}\otimes_{\kappa\left(\mathbf P_{\ell}\right)}\kappa(\mathbf H_{2\ell})\longrightarrow \frac{\widehat{\mathbf P}_{\bar\beta,\ell}}{\widehat{\mathbf P}^+_{\bar\beta,\ell}}
\end{equation}
of $\kappa\left(\mathbf H_{2\ell}\right)$-vector spaces.

\begin{remark}\label{principality}
\begin{enumerate}
\item We have
$$
\frac{(R_\nu)_{\mathbf m_\ell}}{\mathbf m_{\ell-1}(R_\nu)_{\mathbf m_\ell}}=R_{\nu_\ell}.
$$
The ideal $\mathbf P_{\bar\beta}\subset\frac{(R_\nu)_{\mathbf m_\ell}}{\mathbf m_{\ell-1}(R_\nu)_{\mathbf m_\ell}}$ is the $\nu_\ell$-ideal of value $\bar\beta$; in particular,
$\mathbf P_{\bar\beta}$ is principal, generated by any element $x\in\frac{(R_\nu)_{\mathbf m_\ell}}{\mathbf m_{\ell-1}(R_\nu)_{\mathbf m_\ell}}$ such that
$\nu_\ell(x)=\bar\beta$.
\item By definitions (see \eqref{eq:validealmain}), we have
\begin{equation}
\widehat{\mathbf P}_{\bar\beta,\ell}=\mathbf P_{\bar\beta}\frac{\widehat{\bf R}_{\mathbf H_{2\ell}}}{\mathbf H_{2\ell-1}\widehat{\bf R}_{\mathbf H_{2\ell}}}.\label{eq:extendedprincipal}
\end{equation}
Hence the ideal $\widehat{\mathbf P}_{\bar\beta,\ell}$ is also principal.
\end{enumerate}
\end{remark}

By Remark \ref{principality}, both $\kappa(\mathbf H_{2\ell})$-vector spaces in \eqref{eq:limitisomorphism} are one-dimensional.
 
\begin{corollary}\label{isomorphism} The map (\ref{eq:limitisomorphism}) (and hence also \eqref{eq:treeisomorphism}) is an isomorphism.
\end{corollary}
\begin{corollary}\label{integraldomain} The graded algebra
$$
\gr_{\nu_\ell}\frac{(R_\nu)_{\mathbf m_\ell}}{\mathbf m_{\ell-1}(R_\nu)_{\mathbf m_\ell}}\otimes_{\kappa(\mathbf m_{\ell})}\kappa(\mathbf
H_{2\ell})\cong\bigoplus\limits_{\bar\beta\in\left(\frac{\Delta_{\ell-1}}{\Delta_\ell}\right)_+}\frac{\widehat{\mathbf P}_{\bar\beta ,\ell}}{\widehat{\mathbf
P}^+_{\bar\beta,\ell}}=\gr_{\widehat\nu_{2\ell}}\frac{\widehat{\bf R}_{\mathbf H_{2\ell}}}{\mathbf H_{2\ell-1}\widehat{\bf R}_{\mathbf H_{2\ell}}}
$$
is an integral domain.
\end{corollary}

The extension $\widehat\mu_{2\ell}$ of $\mu_\ell$ to $\frac{\widehat{\bf R}}{\mathbf H_{2\ell-1}}$ is defined by
$\widehat\mu_{2\ell}=\widehat\nu_{2\ell}\circ\widehat\mu_{2\ell+2}$. This completes the definition of $\widehat\mu_{2\ell}$, $\ell\in\{1,\dots,r\}$, by descending recursion on
$\ell$.

\begin{lemma}\label{scalewise-birational} The natural inclusion
\begin{equation}
\gr_{\mu_\ell}\frac R{P_{\ell-1}}\hookrightarrow\gr_{\widehat\mu_{2\ell}}\frac{\widehat R}{\tilde H_{2\ell-1}}\label{eq:grbirational}
\end{equation}
of graded algebras is scalewise-birational.
\end{lemma}
\begin{proof} In the proof that follows we will keep in mind the following commutative diagram
\begin{equation}\label{eq:diagram}
\xymatrix{
&&&\gr_{\mu_\ell}\frac{R'}{P'_{\ell-1}}&\hookrightarrow&\gr_{\widehat\mu_{2\ell}}\frac{\widehat{R'}}{\tilde H'_{2\ell-1}}\\
&\gr_{\mu_{\ell+1}}\frac R{P_\ell}&\hookrightarrow&\gr_{\mu_{\ell+1}}\frac {R'}{P'_\ell}\ar@{^{(}->}[u]&\hookrightarrow&\gr_{\widehat\mu_{2\ell+2}}\frac{\widehat{R'}}{\tilde H'_{2\ell}}\ar@{^{(}->}[u]}
\end{equation}
of natural inclusions, valid for all $R'$ in $\Tt$. Here the lower row is a composition of two birational injections of graded algebras, where the second inclusion is scalewise-birational by the induction hypothesis, and the top row is an inclusion of graded algebras whose scalewise birationality we want to prove.

Take a homogeneous element $\bar x\in\gr_{\widehat\mu_{2\ell}}\frac{\widehat R}{\tilde H_{2\ell-1}}$ and let $x$ be a representative of $\bar x$ in $\frac{\widehat R}{\tilde H_{2\ell-1}}$. Let $\bar\beta=\widehat\mu_{2\ell}(x)$. Take an element $y\in\mathbf P_{\bar\beta}$ and $u\in\frac{\widehat{\bf R}_{\mathbf H_{2\ell}}}{\mathbf
H_{2\ell-1}\widehat{\bf R}_{\mathbf H_{2\ell}}}\setminus\mathbf H_{2\ell}\frac{\widehat{\bf R}_{\mathbf H_{2\ell}}}{\mathbf H_{2\ell-1}\widehat{\bf R}_{\mathbf H_{2\ell}}}$ such that $x=yu^{-1}$ (the existence of $y$ and $u$ follows from Remark \ref{principality} (ii)).

Let $R'$ be such that $y\in\frac{R'_{P_\ell}}{P'_{\ell-1}R'_{P_\ell}}$ and $u\in\frac{\widehat R'_{\tilde H'_{2\ell}}}{\tilde H'_{2\ell-1}\widehat R'_{\tilde H'_{2\ell}}}\setminus\tilde H'_{2\ell}\frac{\widehat R'_{\tilde H'_{2\ell}}}{\tilde H'_{2\ell-1}\widehat R'_{\tilde H'_{2\ell}}}$. Write $y=\frac{y_1}{y_2}$ with $y_1,y_2\in\frac{R}{P_{\ell-1}}$. Write $u$ as $u=\frac ws$, where
\begin{equation}
w,s\in\frac{\widehat R'}{\tilde H'_{2\ell-1}}\setminus\frac{\tilde H'_{2\ell}}{\tilde H'_{2\ell-1}}\label{eq:wsnotin}
\end{equation}
for some $R'$ in $\Tt$.

Let $\bar\ $ denote the operation of taking the natural image of an element of $\frac{\widehat{\bf R}}{\mathbf H_{2\ell-1}}$ in $\gr_{\widehat\mu_{2\ell}}\frac{\widehat{\bf
R}}{\mathbf H_{2\ell-1}}$, its initial form with respect to the filtration defined by $\widehat\mu_{2\ell}$. The inclusion \eqref{eq:wsnotin} implies that $\bar w,\bar s\in\gr_{\widehat\mu_{2\ell+2}}\frac{\widehat{R'}}{\tilde H'_{2\ell}}$. In view of the birationality of the second row of diagram \eqref{eq:diagram}, there exist $\bar w_1,\bar w_2,\bar s_1,\bar s_2\in\gr_{\mu_{\ell+1}}\frac {R}{P_\ell}$ satisfying $\bar w=\frac{\bar w_1}{\bar w_2}$ and $\bar s=\frac{\bar s_1}{\bar s_2}$.

Putting together all of the above, we obtain
$$
\bar x(\bar y_2\bar w_1\bar s_2)\in\gr_{\mu_\ell}\frac{R}{P_{\ell-1}}
$$
and $\bar y_2\bar w_1\bar s_2\in\gr_{\mu_\ell}\frac {R}{P_{\ell-1}}$, the latter containment implying that
$$
ord(\bar y_2\bar w_1\bar s_2)\in\Delta_{\ell-1}.
$$
This completes the proof of  the lemma.
\end{proof}

We put $\hat\nu_-=\widehat\mu_2=\hat\nu_2\circ\dots\circ\hat\nu_{2r}$, where for each $\ell$ the valuation $\hat\nu_{2\ell}$ is centered in the local ring
$\frac{\widehat{\mathbf R}_{\mathbf H_{2\ell}}}{\mathbf  H_{2\ell-1}\widehat{\mathbf R}_{\mathbf H_{2\ell}}}$.
\medskip

This completes the construction of the valuation $\hat\nu_-$.

The scalewise birationality of the extension $\hat\nu_-$ of $\nu$ to the ring $\frac{\widehat R}{\tilde H_1}$ is given by Lemma \ref{scalewise-birational}, applied with $\ell=1$. This completes the proof of Theorem \ref{teissier}.\hfill$\Box$

\affiliationone{
   F. J. Herrera Govantes and M. A. Olalla Acosta\\
   Departamento de \'Algebra\\
   Facultad de Matem\'aticas\\
   Calle Tarfia s/n\\
   41012 Sevilla Spain
   \email{jherrera@algebra.us.es\\
   miguelolalla@us.es}}
\affiliationtwo{
   M. Spivakovsky\\
   Institut de Math\'ematiques de Toulouse\\
   UMR 5219 du CNRS,\\
   Universit\'e Paul Sabatier\\
   118, route de Narbonne\\
   31062 Toulouse cedex 9, France.
   \email{mark.spivakovsky@math.univ-toulouse.fr}}
\affiliationthree{
   B. Teissier\\
   Universit\'e Paris Cit\'e and Sorbonne Universit\'e, CNRS, IMJ-PRG, F-75013 Paris, France

 \email{bernard.teissier@imj-prg.fr}}


\begin{thebibliography}{99}
\bibitem{A} S. Abhyankar, {\em Local uniformization on algebraic surfaces over ground fields of characteristic $p\ne0$}, Ann. of Math., 63 (1956) 491--526.
\bibitem{Cut} S. D. Cutkosky, {\em Extensions of Valuations to the Henselization
and Completion}, Acta Math. Vietnam (2019) 44:159--172.
\bibitem{CE} S. D. Cutkosky and S. El Hitti, {\em Formal prime ideals of infinite value and their algebraic resolution}, Ann. Fac. Sci. Toulouse Math. (6) 19 (2010), no. 3-4, 635--649.
\bibitem{CG} S. D. Cutkosky and L. Ghezzi, {\em Completions of valuation rings}, Contemp. Math. 386 (2005), 13--34.
\bibitem{CT} S.D. Cutkosky and B. Teissier, {\em Semigroups of valuations on local rings}, Michigan Math. J., Vol. 57 (2008), 173-194.
\bibitem{EGA} A. Grothendieck, J. Dieudonn\'e, El\'ements de G\'eom\'etrie alg\'ebrique, Chap. IV, Pub. Math. IHES, No.
24, 1965.
\bibitem{HeSa} W. Heinzer and J. Sally, {\em Extensions of valuations to the completion of a local domain.} J. Pure Appl. Algebra, Vol. 71, no 2--3, pp. 175--186, (1991).
\bibitem{HOS} J. Herrera, M.A. Olalla, M. Spivakovsky, {\em Valuations in algebraic field extensions}, J. Algebra 312 (2007), no. 2, 1033--1074.
\bibitem{HOST1} J. Herrera, M. A. Olalla, M. Spivakovsky, B. Teissier, {\em Extending a valuation centered in a local domain to its formal completion}, Proc. London Math. Soc. (3) 105 (2012) 571--621,
\bibitem{HOST2} J. Herrera, M. A. Olalla, M. Spivakovsky, B. Teissier, {\em Extending valuations to formal completions}, in ``Valuation theory in interaction",  Proceedings of the second international conference on valuation theory, Segovia--El Escorial, 2011. Edited by A. Campillo, F-V. Kuhlmann and B. Teissier. European Math. Soc. Publishing House, Congress Reports Series, Sept. 2014, pp 252--265.
\bibitem{L} J. Lipman, {\em Desingularization of two-dimensional schemes},
Ann. Math. 107 (1978) 151--207.
\bibitem{Mat} H. Matsumura, {\em Commutative Algebra.} Commutative algebra. Second edition. Mathematics Lecture Note Series, 56. Benjamin/Cummings Publishing Co., Inc., Reading, Mass., 1980.
\bibitem{Nag} M. Nagata, {\em Local Rings.}, Interscience Publishers, 1960.
\bibitem{Spi1} M. Spivakovsky, {\em Valuations in function fields of surfaces.} Amer. J. Math., 112, 1, 107--156 (1990).
\bibitem{Spi2} M. Spivakovsky, {\em Resolution of singularities~I: Local Uniformization of an equicharacteristic quasi-excellent local domain}, in preparation.
\bibitem{Te} B. Teissier, {\em Valuations, deformations, and toric geometry},  Proceedings of the Saskatoon Conference and Workshop
on valuation theory (second volume), F-V. Kuhlmann, S. Kuhlmann, M. Marshall, editors,  Fields Institute Communications, {\bf 33},
2003, 361-459.
\bibitem{Te2} B. Teissier, {\em Overweight deformations of affine toric varieties and local uniformization}, in "Valuation theory in interaction",  Proceedings of the second international conference on valuation theory, Segovia--El Escorial, 2011. Edited by A. Campillo, F-V. Kuhlmann and B.Teissier. European Math. Soc. Publishing House, Congress Reports Series, Sept. 2014, 474--565.
\bibitem{V1} M. Vaqui\'e, {\em Valuations}, {\it in} ``Resolution of Singularities, a research textbook in tribute to Oscar Zariski'', Birkh\"auser, Progress in Math. No. 181, 2000.
\bibitem{V2} M. Vaqui\'e, {\em Famille admissible de valuations et  d\'efaut d'une extension}, J. Algebra 311 (2007), no. 2, 859--876.
\bibitem{Z} O. Zariski, {\em Local uniformization theorem on algebraic varieties}, Ann. of Math., 41 (1940), 852--896.
\bibitem{ZS} O. Zariski, P. Samuel {\em Commutative Algebra, Vol. II}, Springer-Verlag (1960).
\end{thebibliography}
\end{document}